\documentclass[11pt]{article}
\usepackage{a4}
\usepackage[T1]{fontenc}
\usepackage[english]{babel}
\usepackage{latexsym}
\usepackage{amsmath}
\usepackage{amssymb}
\usepackage{amsthm}
\usepackage{mathrsfs}
\usepackage{color}
\usepackage{cite}
\usepackage{titling}

\newtheorem{teo}{Theorem}[section]

\newtheorem{rem}{Remark}[section]

\newtheorem{pro}{Proposition}[section]
\newtheorem{lem}{Lemma}[section]
\newtheorem{cor}{Corollary}[section]
\usepackage{fancyhdr}

\fancypagestyle{plain}{\fancyhead{}}
\title{\textbf {Uniformly polynomially stable approximations for a class of second order evolution equations}}
\author{{Hajjej Zayd}\\ \textit{D\'epartement de Math\'ematiques, Facult\'e des Sciences de Monastir,}\\ \textit{Université de Monastir, 5019 Monastir, Tunisie.}\\ \textit{Unité de Recherche: Analyse et Contr\^ole des \'Equations aux Dérivées }\\ \textit{Partielles (Code : MESRS  05/UR/15-01)}\\
\textit{email: hajjej.zayd@gmail.com}}
\date{}
\begin{document}
\maketitle
 \abstract{}
 In this paper we study  time semi-discrete approximations of a class of polynomially stable infinite dimensional systems modeling the damped vibrations. We prove that adding a suitable numerical viscosity term in the numerical scheme, one obtains approximations that are uniformly polynomially stable with respect to the discretization parameter.\\\\\




\section{Introduction}
Let $X$ and $Y$ be real Hilbert spaces ( $Y$ will be identified to its dual space) with norms denoted respectively by $\Vert .\Vert_X$ and $\Vert .\Vert_Y$.\\
Let $\mathcal{A} : D(\mathcal{A})\to X$ be a self-adjoint positive operator with $\mathcal{A}^{-1}$ compact in $X$ and let $\mathcal{B}\in\mathfrak{L}(Y,X)$.
 We consider the system described by
\begin{equation}\label{1}
\left \{
\begin{array}{lcr}
\ddot{w}(t)+\mathcal{A}w(t)+\mathcal{B}\mathcal{B}^{*}\dot{w}(t)=0,\;\;t\geq 0\;\;\; \\
w(0)=w_0,\; \dot{w}(0)=w_1.
\end{array}
\right.
\end{equation}
Most of the linear equations modeling the damped vibrations of elastic structures can be written in the form (\ref{1}).\\
We define the energy of solutions at instant $t$ by
\begin{equation}\label{2}
E(t)=\frac{1}{2}\left\{\Vert \dot{w}(t)\Vert_X^2+\Vert \mathcal{A}^{\frac{1}{2}}w(t)\Vert_X^2\right\},
\end{equation}
which satisfies
\begin{equation}\label{3}
\frac{dE}{dt}(t)=-\Vert \mathcal{B}^*\dot{w}(t)\Vert_Y^2,\;\;\;\;\; \forall t\geq 0.
\end{equation}
It is well known that the natural well-posedness space for (\ref{1}) is \\
$H=V\times X$ where $V=D(\mathcal{A}^{\frac{1}{2}})$ and $\Vert x\Vert_V=\Vert \mathcal{A}^{\frac{1}{2}}x\Vert_X, \forall x\in V$.\\
The existence and uniqueness of finite energy solutions of (\ref{1}) can be obtained by standard semigroup methods.\\

We consider the undamped system associated to (\ref{1}):
\begin{equation}\label{4}
\left \{
\begin{array}{lcr}
\ddot{\phi}(t)+\mathcal{A}\phi(t)=0,\;\;t\geq 0\;\;\; \\
\phi(0)=w_0,\; \dot{\phi}(0)=w_1.
\end{array}
\right.
\end{equation}

We assume that system (\ref{4}) satisfy a "weakened" observability inequality (that is the case when the damping operator is effective on a subdomain
where the Geometric Control Condition is not fulfilled [4]), that is there exist positive constants $T, C>0$ and $\beta>-\frac{1}{2}$ such that for all $(w_0, w_1)\in D(\mathcal{A})\times V$ we have
\begin{equation}\label{5}
\int_{0}^{T} \Vert \mathcal{B}^{*}\phi'(t)\Vert_Y^2 dt \geq C \Vert (w_0, w_1)\Vert_{X_{-\beta}\times X_{-\beta-\frac{1}{2}}}^2,
\end{equation}
where
\begin{equation*}
\left \{
\begin{array}{lcr}
            X_{\beta} = D(\mathcal{A}^{\beta}), \;\; \beta \geq 0,\\

             X_{-\beta}= (D({\mathcal{A}^{\beta}}))',\;\;\; \beta \geq 0.
\end{array}
\right.
\end{equation*}
The dual space is obtained by means of the inner product in $X$.\\
Then system (\ref{1}) is polynomially stable \cite{2}, that is there exist a constant $C_1>0$ such that for all $t>0$ and for all $(w_0, w_1)\in D(\mathcal{A})\times V$ we have
$$ E(t)\leq \frac{C_1}{t^{\frac{1}{2\beta+1}}}\Vert (w_0, w_1)\Vert_{D(\mathcal{A})\times V}^2.$$

Our goal is to get, as consequence of (\ref{5}), polynomial stability results for time-discrete systems.

If we introduce $z(t):= \left( \begin{array}{c}
w(t) \\
\dot{w}(t)
\end{array} \right),$ \;\;\;\; $y(t)=\left( \begin{array}{c}
\phi(t) \\
\dot{\phi}(t)
\end{array} \right),$\\

then $z$ satisfies
$$\dot{z}(t)= \left( \begin{array}{c}
\dot{w}(t) \\
-\mathcal{A}w(t)-\mathcal{B}\mathcal{B}^*\dot{w}(t)
\end{array} \right).$$
Consequently the problem (\ref{1}) may be rewritten as the first order evolution equation
\begin{equation}\label{6}
\left \{
\begin{array}{lcr}
            \dot{z}(t) =A z(t)-B B^* z(t),\\

             z(0)= z_0=(w_0, w_1),
\end{array}
\right.
\end{equation}

where $A:D(A)\to H , A=\left(\begin{array}{c}
0\;\;\;\;\;\;\;\;\;\;I\\
-\mathcal{A}\;\;\;\;\;\;0
\end{array} \right),$\;\;\;$B=\left( \begin{array}{c}
0 \\
\mathcal{B}
\end{array} \right),$\\
$B^*=(0,\mathcal{B}^*)$ ,\;\; $D(A)= D(\mathcal{A})\times V$ and $H=V\times X$.\\

With this notation, (\ref{5}) becomes
\begin{equation}\label{8}
\int_{0}^{T} \Vert B^*\; y(t)\Vert_Y^2 dt \geq C \Vert z_0 \Vert_{X_{-\beta}\times X_{-\beta-\frac{1}{2}}}^2.
\end{equation}
\\\\\\

In recent years an important literature was devoted to the space/time semi-discrete
approximations of a class of exponentially stable infinite dimensional
systems. Let us also mention the recent work [1], where polynomial stability was discussed for space discrete schemes of (\ref{1}). It has been proved that exponential/polynomial
stability may be lost under numerical discretization as the mesh size tends
to zero due to the existence of high-frequency spurious solutions. \\

Several remedies have been proposed and
analyzed to overcome this difficulties. Let us quote the Tychonoff regularization [11, 22, 21, 23, 8, 1 ],
a bi-grid algorithm [9, 19], a mixed finite element method [10, 3, 5, 6, 18], or filtering the high
frequencies [13, 16, 25, 7, 24].
As in [1, 8, 21, 22, 23] our goal is to damp the spurious high frequency modes by introducing a numerical
viscosity in the approximation schemes. Though our paper is inspired from \cite{8}, it differs from that paper on the following points:\\

i)\; We analyze the polynomial decay of the discrete schemes when the continuous problem has such a decay.\\

ii)\; For the proof of the discrete observability inequality, we will use a  method  based on a decoupling argument of low and high frequencies, the low frequency observability property for time semi-discrete approximations of conservative linear systems and the dissipativity of the numerical viscosity on the high frequency components. But, for the low frequency, contrary to \cite{8} where a Hautus-type test is required, we use a spectral approach and a discrete Ingham  type inequalities ( when the spectrum of the spatial operator $A$ associated
with the undamped problem satisfies such a gap condition). \\

Note however that we cannot apply these methods when the damped operator $B$ is not bounded, as in [2], where
the wave equation is damped by a feedback law on the boundary. Dealing with unbounded damping operators $B$ needs
further work.\\

Despite all the existing literature, this article seems to be the first one to provide a systematic way of transferring
polynomial decay properties from the continuous to the time-discrete setting.\\

The paper is organized as follows. In section 2, we prove a uniform "weakened" observability after the addition of numerical viscosity term by using, as we said, a decoupling argument and a spectral approach. Section 3 is devoted to prove the main result of this paper. We illustrate our results by presenting different examples in Section 4. Finally, some  further comments and open problems are collected in section 5.\\

In the following, we will write $E\sim F$ instead of $c_1 E \leq F \leq c_2 E$ for brevity, where $c_1, c_2>0$ are constants.\\

\section{observability of time-discrete systems}

In this section, we assume that system (\ref{6}) is polynomially stable and $B^*\in\mathfrak{L}(H,Y)$, i.e. there exists a constant $K_B$ such that

$$\Vert B^*z\Vert_Y\leq K_B\Vert z\Vert_H,\;\;\;\forall\;z\in H.$$

We start considering the following  time-discretization scheme for the continuous system (\ref{1}) or equivalent for the system (\ref{6}).
For any $\Delta t>0$, we denote by $z^k$ the approximation of the solution $z$ of system (\ref{6}) at time $t_k=k\Delta t$, for $k\in\mathbb{N}$, and we consider time discretization of system (\ref{6}):

\begin{equation}\label{9}
\left \{
\begin{array}{lcr}
            \frac{\tilde{z}^{k+1}-z^k}{\Delta t} = A\left(\frac{z^k+\tilde{z}^{k+1}}{2}\right)-BB^*\left(\frac{z^k+\tilde{z}^{k+1}}{2}\right),\;k\in \mathbb{N}, & \\\\

             \frac{z^{k+1}-\tilde{z}^{k+1}}{\Delta t}={(\Delta t)^2}{A^2}{z^{k+1}} ,\;k\in\mathbb{N},\\\\

             z^0=z_0.
\end{array}
\right.
\end{equation}

The numerical viscosity term $(\Delta t)^2A^2$ in (\ref{9}) is introduced in order to damp the high frequency modes.\\

We can define the discrete energy by:
\begin{equation}\label{10}
E^k=\frac{1}{2}\Vert z^k\Vert_H^2,\;\;k\geq 0.
\end{equation}
The energy satisfies (\cite{8}):

\begin{equation*}
E^{k+1}+(\Delta t)^3\left\Vert Az^{k+1}\right\Vert_H^2+\frac{(\Delta t)^6}{2}\left\Vert {A^2}{z}^{k+1}\right\Vert_H^2+ \Delta t \left\Vert B^*\left(\frac{z^k+\tilde{z}^{k+1}}{2}\right)\right\Vert_Y^2 = E^k.
\end{equation*}
Summing from $j=0$ to $l=[T/\Delta t]$, it follows then that :\\

$E^0-E^{l+1}=\Delta t\displaystyle{\sum_{j=0}^{l}}(\Delta t)^2\left\Vert Az^{j+1}\right\Vert_H^2+\frac{\Delta t}{2}\displaystyle{\sum_{j=0}^{l}}(\Delta t)^5 \left\Vert {A^2}{z^{j+1}}\right\Vert_H^2$
\begin{equation}\label{11}
 \hspace{4cm} +\;\Delta t\displaystyle{\sum_{j=0}^{l}}\left\Vert B^*\left(\frac{z^j+\tilde{z}^{j+1}}{2}\right)\right\Vert_Y^2.
\end{equation}

Note that this numerical scheme is based on the decomposition of the operator $A-BB^*+(\Delta t)^2A^2$ into its
conservative and dissipative parts, that we treat differently. Indeed, the midpoint scheme is appropriate for conservative
systems since it preserves the norm conservation property. This is not the case for dissipative systems, since midpoint
schemes do not preserve the dissipative properties of high frequency solutions. Therefore, we rather use an implicit
Euler scheme, which efficiently preserves these dissipative properties.\\\\

The convergence of the solutions of (\ref{9}) towards those of the original system (\ref{6}) when $\Delta t \to 0$ holds in a suitable
topology (\cite{8}).\\
As the continuous level, we will prove that the uniform polynomial decay of system (\ref{6}) is a consequence of the following "weakened" observability inequality for every solution of the following time-discrete system:
\begin{equation}\label{12}
\left \{
\begin{array}{lcr}
            \frac{\tilde{u}^{k+1}-u^k}{\Delta t} = A\left(\frac{u^k+\tilde{u}^{k+1}}{2}\right),\;k\in \mathbb{N}, & \\\\

              \frac{u^{k+1}-\tilde{u}^{k+1}}{\Delta t}={(\Delta t)^2}{A^2}{u^{k+1}},\;k\in\mathbb{N},\\\\

            u^0=u_0.
 \end{array}
\right.
\end{equation}

 We want to show that there exist positive constants $T$, $c$ and $\beta>-\frac{1}{2}$ such that, for any $\Delta t>0$ every solution $u^k$ of (\ref{12}) satisfies:\\

$ c \Vert u^0\Vert_{X_{-\beta}\times X_{-\beta-\frac{1}{2}}}^2\leq \Delta t \displaystyle{\sum_{k\Delta t\in [0,T]}}\left\Vert B^*\left(\frac{\tilde{u}^{k+1}+u^k}{2}\right)\right\Vert_Y^2$\\

$\hspace{4cm} +\;\Delta t \displaystyle{\sum_{k\Delta t\in [0,T]}}(\Delta t)^2\left\Vert Au^{k+1}\right\Vert_H^2$\\
 \begin{equation}\label{13}
\hspace{5cm}+\;\Delta t\displaystyle{\sum_{k\Delta t\in [0,T]}}(\Delta t)^5\left\Vert {A^2}{u}^{k+1}\right\Vert_H^2.
\end{equation}
Here and in the sequel $c$ denotes a generic positive constant that may vary from line to line but is independent of $\Delta t$.\\

In this section, we show how to obtain the observability inequality (\ref{13}). Before giving spectral conditions to obtain  polynomial decay, we need to introduce some notations.\\\\

Since $A$ is a skew-adjoint operator with compact resolvent, its spectrum is discrete and $\sigma(A)=\{i\mu_j:\;j\in\mathbb{Z^*}\}$ where $(\mu_j)_{j\in\mathbb{Z^*}}$ is a sequence of real numbers such that $\vert \mu_j\vert\to\infty$  when $j\to\infty$. Set
$(\varphi_j)_{j\in\mathbb{Z^*}}$ an orthonormal basis of eigenvectors of $A$ associated to the eigenvalues $(i\mu_j)_{j\in\mathbb{Z^*}}$, that is\\
$A\varphi_j=i\mu_j\varphi_j$, with $\mu_j=
\left \{
\begin{array}{lcr}
            \sqrt{\eta_j}\;\;\text{if}\;j\in\mathbb{N}^*,\\

            -\sqrt{\eta_{-j}}\;\;\text{if}\;(-j)\in\mathbb{N}^*,
\end{array}
\right.
$ \\
and \;\;$\varphi_j=\frac{1}{\sqrt{2}}\left(\begin{array}{l}
\frac{1}{i\mu_j}\phi_j\\
\phi_j
\end{array}\right)$\;\;\;$\forall\; j\in\mathbb{Z}^*$,\;\;\;\; (  we define\\
 $\phi_{-j}=\phi_j,\;\forall j\in\mathbb{N}^*$)\\

where $\eta_j$ and $\phi_j$ are the  eigenvalues and the corresponding eigenvectors of $\mathcal{A}$.\\

Moreover, define

$$C_s = span \{\varphi_j:\;\text{the corresponding}\;i\mu_j\;\text{satisfies}\;\vert\mu_j\vert\leq s\}.$$
\\

Now, we recall some results about Discrete Ingham type inequalities.\\

\begin{teo}(\cite{14}) Assume that there exist a positive number $\gamma$ satisfying
\begin{equation}\label{14}
\vert\omega_{k}-\omega_n\vert \geq \gamma \;\; \text{for all}\; k\neq n,
\end{equation}
where $(\omega_k)_{k\in\mathbb{Z}}$ is a family of real numbers.\\
Given $0<\sigma\leq \pi/\gamma$ arbitrarily, fix an integer $J$ such that $J\sigma > \pi/\gamma$. Then there exist two positive constants $c_1$ and $c_2$, depending only on $\gamma$ and $J\sigma$, such that, for every $t\in\mathbb{R}$, we have
\begin{equation}\label{15}
\sigma \displaystyle{\sum_{j=-J}^{J}}\left\vert\sum_{k\in\mathbb{Z}} x_k e^{i\omega_k(t+ j\sigma)}\right\vert^2\sim \sum_{k\in\mathbb{Z}}\vert x_k\vert^2,
\end{equation}
with complex coefficients $x_k$ satisfying the condition
\begin{equation}\label{16}
 x_k=0 \;\;\text{whenever}\;\; \vert w_k\vert \geq \frac{\pi}{\sigma}-\frac{\gamma}{2},
\end{equation}
\end{teo}

\begin{teo}(\cite{15}) Assume that there exist a positive number $\gamma_1$ satisfying
\begin{equation}\label{17}
\omega_{k+2}-\omega_k\geq 2\gamma_1 \;\; \forall\; k,
\end{equation}
where $(\omega_k)_{k\in\mathbb{Z}}$ is a family of real numbers.\\
Given $0<\sigma_1\leq \pi/\gamma_1$ arbitrarily, fix an integer $J$ such that $J\sigma_1 > \pi/\gamma_1$. Then there exist two positive constants $c_3$ and $c_4$, depending only on $\gamma_1$ and $J\sigma_1$,
\begin{equation}\label{18}
c_3 Q(x)\leq \sigma_1 \displaystyle{\sum_{j=-J}^{J}}\left\vert\sum_{k\in\mathbb{Z}} x_k e^{i\omega_k j\sigma_1}\right\vert^2\leq c_4 Q(x),
\end{equation}
with complex coefficients $x_k$ satisfying the condition
\begin{equation}\label{19}
 x_k=0 \;\;\text{whenever}\;\; \vert w_k\vert \geq \frac{\pi}{\sigma_1}-\frac{\gamma_1}{2},
\end{equation}
and where
$$ Q(x)=\sum_{k\in \mathbb{A}_1}\vert x_k\vert^2+\sum_{k\in \mathbb{A}_2}\vert x_k + x_{k+1}\vert^2 + (\omega_{k+1}-\omega_k)^2(\vert x_k\vert^2+\vert x_{k+1}\vert^2).$$
\end{teo}

The above equivalence means that ( see \cite{20} for more details)
$$ \sigma_1 \displaystyle{\sum_{j=-J}^{J}}\left\vert\sum_{k\in\mathbb{Z}} x_k e^{i\omega_k j\sigma}\right\vert^2\geq c_3\sum_{l=1}^{2}\sum_{k\in \mathbb{A}_l} \Vert B_{k}^{-1}C_{k}\Vert_2^2,$$
where $\Vert .\Vert_2$ means the Euclidean norm of the vector, for $k\in \mathbb{A}_l$ the vector $C_{k}$ and the $l\times l$ matrix $B_{k} $ are given by\\

$$C_{k}=x_k,\; B_{k}^{-1}=1\;\;\text{if}\;\; l=1,$$
and
$$C_{k}=\left( \begin{array}{c}
x_k \\
x_{k+1}
\end{array} \right),\;B_{k}^{-1}=\left( \begin{array}{c}
1\;\;\;\;\;\;\;\;\;\;\;\;\;\;\;\; 1\\
0 \;\;\;\;\;\;\;\;\;\;\;\omega_{k+1}-\omega_k\end{array} \right)\;\;\;\text{if}\;\; l=2.$$

Now, as in \cite{20}, let $U$ be a separable Hilbert space (in the sequel, $U$ will be $Y$ ). For a vector $d=\left( \begin{array}{c}
d_1\\
.\\
.\\
.\\
d_m\end{array} \right)\in U^m$, we set $\Vert .\Vert_{U,2}$ the norm in $U^m$ defined by
$$\Vert d\Vert_{U,2}^2=\sum_{j=l}^{m}\Vert d_j\Vert_U^2.$$
Let $\delta\in(0,\delta_0)$ where $\delta_0=\min(\pi-\frac{\Delta t \gamma}{2}, \pi-\frac{\Delta t \gamma_1}{2})$.\\

Then we obtain the discrete inequality of Ingham's type in $U$ :\\\\\\

\begin{pro}
If $\mu_n$ satisfy  (\ref{14}), then for all sequence $(a_n)_n$ in $U$, the function
$$x^k=\sum_{\vert \mu_n\vert\leq\delta/\Delta t} a_n e^{ik\Delta t \mu_n}$$
satisfies the estimates
$$\Delta t \sum_{k\Delta t\in[0,T]}\Vert x^k\Vert_U^2\sim\sum_{\vert \mu_n\vert\leq\delta/\Delta t}\Vert a_n\Vert_U^2,$$
for $T>T_0=\frac{2\pi}{\gamma}$.
\end{pro}
\begin{proof}
Since $U$ is a separable Hilbert space, there exists a Hilbert basis $(\psi_j)_{j\geq 1}$ of $U$. Therefore, $a_n\in U$ can be written as
$$a_n=\sum_{j=1}^{+\infty}a_n^j\psi_j.$$
We truncate $a_n$ as follows: for $K\in\mathbb{N}^*$, let $a_n^{(K)}=\displaystyle{\sum_{j=1}^{K}} a_n^j\psi_j$ \\
and set $x_K^k=\displaystyle{\sum_{j=1}^{K}}\left(\displaystyle{\sum_{\vert \mu_n\vert\leq\delta/\Delta t}} a_n^j e^{ik\Delta t \mu_n}\right)\psi_j.$\\
Since $(\psi_j)_{j\geq 1}$ is a Hilbert basis, we have by Parseval's theorem
$$\Vert x_K^k\Vert_U^2=\displaystyle{\sum_{j=1}^{K}}\left\vert\;\;\displaystyle{\sum_{\vert \mu_n\vert\leq\delta/\Delta t}}a_n^j e^{ik\Delta t \mu_n}\right\vert^2.$$
Thus, by applying discrete Ingham type inequality, we have
$$\Delta t \displaystyle{\sum_{k\Delta t\in[0,T]}}\Vert x_K^k\Vert_U^2=\Delta t \displaystyle{\sum_{j=1}^{K}}\;\;\displaystyle{\sum_{k\Delta t\in[0,T]}}\left\vert\displaystyle{\sum_{\vert \mu_n\vert\leq\delta/\Delta t}}a_n^j e^{ik\Delta t \mu_n}\right\vert^2.$$\\

$$\sim\displaystyle{\sum_{j=1}^{K}} \;\;\displaystyle{\sum_{\vert \mu_n\vert\leq\delta/\Delta t}}(a_n^j)^2.$$
$$\sim\displaystyle{\sum_{\vert \mu_n\vert\leq\delta/\Delta t}}\;\;\displaystyle{\sum_{j=1}^{K}}(a_n^j)^2.$$
Therefore
$$\Delta t \displaystyle{\sum_{k\Delta t\in[0,T]}}\Vert x_K^k\Vert_U^2\sim\displaystyle{\sum_{\vert \mu_n\vert\leq\delta/\Delta t}}\Vert a_n^{(K)}\Vert_U^2.$$
Since $x_K^k\to x^k$ and $a_n^{(K)}\to a_n$ when $K\to +\infty$, we obtain the result.
\end{proof}
\begin{cor}
With the same hypothesis of Theorem 2.2, for all sequence $(a_n)_{n\in\mathbb{Z^*}}$ in $U$, the function
$$ f^k=\sum_{n\in\mathbb{Z^*}} a_n e^{i\omega_n k\Delta t},$$
satisfy, for $T>T_1=\frac{2\pi}{\gamma_1}$, the inequality
$$\Delta t \displaystyle{\sum_{k\Delta t\in[0,T]}}\Vert f^k \Vert_U^2\geq c \sum_{l=1}^{2}\sum_{n\in\mathbb{A}_l} \Vert B_{n}^{-1}C_{n}\Vert_{U,2}^2,$$
with $ a_n=0 \;\;\text{whenever}\;\; \vert w_n\vert \geq \frac{\pi}{\Delta t}-\frac{\gamma_1}{2},$ and $c>0$.
\end{cor}
\begin{pro}
Assume that $Y$ is separable. Let $y^k$ the solution of the following system
\begin{equation}\label{20}
\left \{
\begin{array}{lcr}

    \frac{y^{k+1}-y^{k}}{\Delta t}=A\left(\frac{y^{k+1}+y^{k}}{2}\right),\;\; k\in\mathbb{N},\\

             y^{0}=z_{0}=(w_0, w_1).
\end{array}
\right.
\end{equation}
\begin{enumerate}
\item
Assume that $\mu_j$ satisfy (\ref{14}) and  for all $y^0 \in V\times X$ we have
\begin{equation}\label{21}
\exists\; \theta>0,\;\forall\;j\geq 1,\;\Vert B^*\varphi_j\Vert_Y\geq \frac{\theta}{\mu_j^{2\beta+1}}
\end{equation}
for some constant $\theta>0$ and for a fixed real number $\beta> -\frac{1}{2}$.\\

Then, there exist a time $T > T_0$ and a constant $C > 0$ such that

\begin{equation}\label{22}
C\Vert y^{0}\Vert_{X_{-\beta}\times X_{-\beta-\frac{1}{2}}}^2 \leq \Delta t \displaystyle { \sum_{k \Delta t \in \left[0,T \right]}
}\left\Vert B^*\left(\frac{y^k+y^{k+1}}{2}\right)\right\Vert_Y^2,\; \forall\; y^0 \in C_{\delta/ \Delta t}.
\end{equation}

\item
Assume that $\mu_n$ verify (\ref{17}) and, for all $y^0 \in V\times X$,
\begin{equation}\label{23}
\exists\; \theta>0,\;\forall\;l=1,2,\;\;\forall\;n\in \mathbb{A}_l,\;\forall\;\xi\in\mathbb{R}^2,\;\Vert B_n^{-1}\phi_n\xi\Vert_{Y,2}\geq \frac{\theta}{\mu_n^{2\beta+1}}\Vert \xi\Vert_2,
\end{equation}
\\

then there exist a time $T > T_1$ and a constant $C > 0$ such that (\ref{22}) holds true.\\

\end{enumerate}

\end{pro}

\begin{rem}
In the last proposition, we have chosen $\delta/ \Delta t$ the filtering parameter. Indeed, this scale is linked with the paper \cite{8}. The question of  optimality of this choice, in our case, remains open.
\end{rem}

\begin{proof}
We first show that (\ref{21})$\Longrightarrow$ (\ref{22}).\\
Simple formal calculations give
$$y^{k+1}=(I+\frac{\Delta t}{2}A)^{-1}(I-\frac{\Delta t }{2}A)y^k$$
$\hspace{3.8cm}=e^{i\alpha_j\Delta t}y^k,$\\

where $e^{i\alpha_j\Delta t}=\frac{1+\frac{\Delta t}{2}i\mu_j}{1-\frac{\Delta t}{2}i\mu_j}.$\\
Writing
$$y^0=\displaystyle{\sum_{\vert \mu_j\vert\leq\delta/\Delta t}}c_j \varphi_j=\left(\begin{array}{c} w_0\\
w_1\end{array}\right)=\frac{1}{\sqrt{2}}\left(\begin{array}{c} \displaystyle{\sum_{\vert \mu_j\vert\leq\delta/\Delta t}}\frac{1}{i\mu_j}c_j\phi_j\\
\displaystyle{\sum_{\vert \mu_j\vert\leq\delta/\Delta t}}c_j\phi_j\end{array}\right),$$
We have $\Vert w_0\Vert_{X_{-\beta}}^2=\Vert w_1\Vert_{X_{-\beta-\frac{1}{2}}}^2\sim\frac{1}{2}\displaystyle{\sum_{\vert \mu_j\vert\leq\delta/\Delta t}} c_j^2\eta_j^{-2\beta-1}.$\\

The solution $y^k$  is given by
$$y^k=\displaystyle{\sum_{\vert \mu_j\vert\leq\delta/\Delta t} }c_j e^{i\alpha_j k\Delta t}\varphi_j,$$
where
$$ \alpha_j=\frac{2}{\Delta t}\arctan(\frac{\mu_j\Delta t}{2}).$$
Consequently
$$B^*(\frac{y^{k+1}+y^k}{2})=\displaystyle{\sum_{\vert \mu_j\vert\leq\delta/\Delta t}}c_j\cos (\frac{\alpha_j\Delta t}{2})e^{i\alpha_j (k+\frac{1}{2})\Delta t} B^*\varphi_j .$$
It is easy to check that $\vert\alpha_k-\alpha_n\vert\geq \gamma\prime=\frac{\gamma}{2}$ (for all $k\neq n$) for $\Delta t $ sufficiently small, and
$$\cos^2(\frac{\alpha_j\Delta t}{2})=\cos^2(\arctan(\frac{\mu_j\Delta t}{2}))=\frac{1}{1+\frac{(\mu_j\Delta t)^2}{4}}\geq \frac{1}{1+\frac{\delta^2}{4}}.$$
Now, using Ingham's inequality in $Y$, for $T> T_0$,  we get\\

$$\Delta t \displaystyle { \sum_{k \Delta t \in \left[0,T \right]}
}\left\Vert B^*\left(\frac{y^k+y^{k+1}}{2}\right)\right\Vert_Y^2 \geq C_1 \displaystyle\sum_{\vert \mu_j\vert\leq\delta/\Delta t} c_j^2 \Vert \mathcal{B}^*\phi_j\Vert_Y^2.$$\\
By (\ref{21}), we get
$$\Delta t \displaystyle { \sum_{k \Delta t \in \left[0,T \right]}
}\left\Vert B^*\left(\frac{y^k+y^{k+1}}{2}\right)\right\Vert_Y^2 \geq C_1 \displaystyle{\sum_{\vert \mu_j\vert\leq\delta/\Delta t}} c_j^2 \mu_j^{-2(2\beta+1)}$$\\

$$\hspace{4cm}=C_1\left\Vert y^{0}\right\Vert_{X_{-\beta}\times X_{-\beta-\frac{1}{2}}}^2.$$
The proof of (\ref{23})$\Longrightarrow$ (\ref{22}) is similar to the first one but now we use the other discrete Ingham type inequality presented in Corollary 2.1.   \end{proof}

Applying Proposition 2.2, for any $\delta >0$ defined as above, choosing a time  $T^*>T_2=\max (T_0,T_1)$ there exists a positive constant $C=C_{T^*,\delta}$ such that the inequality (\ref{22}) holds for any solution $y^k$ with $y^0\in C_{\delta/\Delta t}$ and where the gap condition (\ref{14}) or (\ref{17}) is satisfied. In the sequel, we fix $T^*=2T_2$.

\begin{lem}
If $\mu_j$ verify  (\ref{14}) or (\ref{17}) , then there exists a constant $c > 0$ such that (\ref{13}) holds with $T=T^*$ for all solutions $u^k$
of (\ref{12}) uniformly with respect to $\Delta t$.
\end{lem}
\begin{proof}
The proof can be done similarly as the one of Lemma 5.2.4. in \cite{8}, we decompose the solution $u^k$ of (\ref{12}) into its low and high frequency parts. To be more precise, we consider
$$u_l^k=p_{\delta/ \Delta t}u^k,\;\;\;u_h^k=(I-p_{\delta/ \Delta t})u^k,$$\\
where $\delta>0$ is the positive number that we have been chosen above, and  $p_{\delta/ \Delta t}u$ is the orthogonal projection on $C_{\delta/ \Delta t}$.\\
Note both $u_l^k$ and $u_h^k$ are solutions of (\ref{12}).\\
In addition, $u_h^k$ lies in the space $ C_{\delta/ \Delta t}^\bot$, in which we have:
\begin{equation}\label{24}
\left \{
\begin{array}{lcr}

            \Delta t \Vert Ay\Vert_H\geq \delta \Vert y\Vert_H,\;\;\forall\;y\in C_{\delta/ \Delta t}^\bot,\;\; \text{and also}\\\\

             \Delta t \Vert Ay\Vert_{X_{-\beta}\times X_{-\beta-\frac{1}{2}}}\geq  \delta \Vert y\Vert_{X_{-\beta}\times X_{-\beta-\frac{1}{2}}},\;\;\; \forall\;y\in C_{\delta/ \Delta t}^\bot.
 \end{array}
\right.
\end{equation}
\\\\\\

{\bf The low frequencies.} First we compare $u_l^k$ with $y_l^k$ solution  of (\ref{20})\\

 with initial data $y_l^k(0)= u_l^k(0)$. Set $w_l^k=u_l^k-y_l^k$. From (\ref{22}), which is\\

 valid for solutions of (\ref{20}) with initial data in $C_{\delta/ \Delta t}$, we get\\

$$C\Vert u_l^{0}\Vert_{X_{-\beta}\times X_{-\beta-\frac{1}{2}}}^2 \leq 2\Delta t \displaystyle { \sum_{k \Delta t \in \left[0,T^* \right]}
}\left\Vert B^*\left(\frac{u_l^k+\tilde{u}_l^{k+1}}{2}\right)\right\Vert_Y^2$$
\begin{equation}\label{25}
\hspace{4.1cm} + 2\Delta t \displaystyle { \sum_{k \Delta t \in \left[0,T^* \right]}
}\left\Vert B^*\left(\frac{w_l^k+\tilde{w}_l^{k+1}}{2}\right)\right\Vert_Y^2.
\end{equation}
Now, we write the equation satisfied by $w_l^k$, which can be deduced from (\ref{12}) and (\ref{20}):
\begin{equation}\label{26}
\left \{
\begin{array}{lcr}
            \frac{\tilde{w}_l^{k+1}-w_l^k}{\Delta t} = A\left(\frac{w_l^k+\tilde{w}_l^{k+1}}{2}\right),\;k\in \mathbb{N}, & \\\\

              \frac{w_l^{k+1}-\tilde{w}_l^{k+1}}{\Delta t}={(\Delta t)^2}{A^2}{u_l^{k+1}},\;k\in\mathbb{N},\\\\

            w_l^0=0.
 \end{array}
\right.
\end{equation}

the energy estimates for $w_l^k$ give:
\begin{equation}\label{27}
\left \{
\begin{array}{lcr}
             \Vert \tilde{w}_l^{k+1}\Vert_H^2=\Vert w_l^k\Vert_H^2,\\\\

             \Vert w_l^{k+1}\Vert_H^2=\Vert \tilde{w}_l^{k+1}\Vert_H^2-2(\Delta t)^3 \left\langle Au_l^{k+1}, A\left(\frac{w_l^{k+1}+\tilde{w}_l^{k+1}}{2}\right)\right\rangle_H.

\end{array}
\right.
\end{equation}

Note that $w_l^k$ and $\tilde{w}_l^{k+1}$ belong to $C_{\delta/ \Delta t}$ for all $k\in\mathbb{N}$, since $u_l^k$ and $y_l^k$ both belong to $C_{\delta/ \Delta t}$. Therefore, the energy estimates for $w_l^k$ lead, for $k\in\mathbb{N}$, to
$$\Vert w_l^k\Vert_H^2=-2\Delta t \sum_{j=1}^{k} (\Delta t)^2 \left\langle Au_l^j, A\left(\frac{w_l^j+\tilde{w}_l^{j+1}}{2}\right)\right\rangle_H$$
$$\hspace{2cm} \leq \Delta t \sum_{j=1}^{k} (\Delta t)^2\Vert Au_l^j\Vert_H^2+\delta^2 \Delta t\sum_{j=1}^{k}\left\Vert\frac{w_l^j+\tilde{w}_l^{j+1}}{2}\right\Vert_H^2$$
$$ \hspace{1cm} \leq \Delta t \sum_{j=1}^{k} (\Delta t)^2\Vert Au_l^j\Vert_H^2+\delta^2 \Delta t\sum_{j=1}^{k}\left\Vert w_l^j\right\Vert_H^2,$$
where we used the first line of (\ref{27}).\\

Gr$\ddot{o}$nwall's Lemma applies and allows to deduced from (\ref{25}) and the fact that the operator $B$ is bounded, the existence of a positive\\
 constant ( that may change from line to line) independent of $\Delta t$ such that\\

$$c\Vert u_l^0\Vert_{X_{-\beta}\times X_{-\beta-\frac{1}{2}}}^2\leq \Delta t \displaystyle { \sum_{k \Delta t \in \left[0,T^* \right]}
}\left\Vert B^*\left(\frac{u_l^k+\tilde{u}_l^{k+1}}{2}\right)\right\Vert_Y^2$$

$$\hspace{2.7cm} +\Delta t \displaystyle { \sum_{k \Delta t \in \left]0,T^* \right]}
} (\Delta t)^2\Vert Au_l^k\Vert_H^2.$$\\

Besides,
$$\Delta t \displaystyle { \sum_{k \Delta t \in \left[0,T^* \right]}
}\left\Vert B^*\left(\frac{u_l^k+\tilde{u}_l^{k+1}}{2}\right)\right\Vert_Y^2 \leq 2\Delta t \displaystyle { \sum_{k \Delta t \in \left[0,T^* \right]}
}\left\Vert B^*\left(\frac{u^k+\tilde{u}^{k+1}}{2}\right)\right\Vert_Y^2$$
$$ \hspace{5.4cm} +2\Delta t \displaystyle { \sum_{k \Delta t \in \left[0,T^* \right]}
}\left\Vert B^*\left(\frac{u_h^k+\tilde{u}_h^{k+1}}{2}\right)\right\Vert_Y^2$$\\

and, since $u_h^k$ and $\tilde{u}_h^{k+1}$ belong to $C_{\delta/ \Delta t}^\bot$ for all $k$, we get from (\ref{24}) that\\

$\Delta t  \displaystyle { \sum_{k \Delta t \in \left[0,T^* \right]}
}\left\Vert B^*\left(\frac{u_h^k+\tilde{u}_h^{k+1}}{2}\right)\right\Vert_Y^2\leq K_B^2\Delta t  \displaystyle {\sum_{k \Delta t \in \left[0,T^* \right]}
}\left\Vert \frac{u_h^k+\tilde{u}_h^{k+1}}{2}\right\Vert_H^2$
$$ \hspace{3cm} \leq K_B^2\Delta t \displaystyle {\sum_{k \Delta t \in \left[0,T^* \right]}}\Vert u_h^k\Vert_H^2$$
$$ \hspace{5cm}\leq \frac{K_B^2}{\delta^2}\Delta t \displaystyle {\sum_{k \Delta t \in \left]0,T^* \right]}}\Vert Au_h^k\Vert_H^2 + K_B^2\Delta t \Vert u_h^0\Vert_H^2,$$
since, from the first line of (\ref{12}),
$$\Vert \tilde{u}_h^{k+1}\Vert_H^2=\Vert u_h^k\Vert_H^2,\;\;\forall\;k\in\mathbb{N}.$$
It follows that there exists $c>0$ independent of $\Delta t$ such that \\

$c\Vert u_l^0\Vert_{X_{-\beta}\times X_{-\beta-\frac{1}{2}}}^2\leq \Delta t \displaystyle { \sum_{k \Delta t \in \left[0,T^* \right]}
}\left\Vert B^*\left(\frac{u^k+\tilde{u}^{k+1}}{2}\right)\right\Vert_Y^2$\\

\begin{equation}\label{28}
\hspace{3.8cm} +\Delta t \displaystyle { \sum_{k \Delta t \in \left]0,T^* \right]}
} (\Delta t)^2\Vert Au_l^k\Vert_H^2+ \Delta t \Vert u_h^0\Vert_H^2.
\end{equation}
\\

{\bf The high frequencies.} We now discuss the decay properties of solutions of (\ref{12}) with initial data $u_h^0\in C_{\delta/ \Delta t}^\bot.$ It is easy to check that for all $k\in\mathbb{N}$, $u_h^k\in C_{\delta/ \Delta t}^\bot$. But, simple calculations give:\\

$\Vert (I-(\Delta t)^3A^2)u_h^{k+1}\Vert_{X_{-\beta}\times X_{-\beta-\frac{1}{2}}}^2$\\

$\hspace{1.2cm} =\Vert u_h^{k+1}\Vert_{X_{-\beta}\times X_{-\beta-\frac{1}{2}}}^2 +2(\Delta t)^3\Vert Au_h^{k+1}\Vert_{X_{-\beta}\times X_{-\beta-\frac{1}{2}}}^2$\\

$\hspace{4cm} +(\Delta t)^6\Vert A^2u_h^{k+1}\Vert_{X_{-\beta}\times X_{-\beta-\frac{1}{2}}}^2$\\
\begin{equation}\label{29}
= \Vert \tilde{u}_h^{k+1}\Vert_{X_{-\beta}\times X_{-\beta-\frac{1}{2}}}^2=\Vert u_h^k\Vert_{X_{-\beta}\times X_{-\beta-\frac{1}{2}}}^2,\;\; k\in\mathbb{N}.
\end{equation}
Due to (\ref{24}), we get:
$$ (1+2(\Delta t)\delta^2)\Vert u_h^{k+1}\Vert_{X_{-\beta}\times X_{-\beta-\frac{1}{2}}}^2\leq \Vert u_h^k\Vert_{X_{-\beta}\times X_{-\beta-\frac{1}{2}}}^2.$$\\

We deduce that
$$ \Vert u_h^{k+1}\Vert_{X_{-\beta}\times X_{-\beta-\frac{1}{2}}}^2\leq \frac{1}{1+2(\Delta t)\delta^2}\Vert u_h^k\Vert_{X_{-\beta}\times X_{-\beta-\frac{1}{2}}}^2,\;\; k\in\mathbb{N},$$\\

which implies
\begin{equation}\label{30}
\Vert u_h^k\Vert_{X_{-\beta}\times X_{-\beta-\frac{1}{2}}}^2\leq\left(\frac{1}{1+2(\Delta t)\delta^2}\right)^k\Vert u_h^0\Vert_{X_{-\beta}\times X_{-\beta-\frac{1}{2}}}^2,\;\;k\in\mathbb{N}.
\end{equation}
Taking $k^*=[T^*/\Delta t]$, we get a constant $\tau<1$ independent of $\Delta t>0$ such that
$$\Vert u_h^{k^*}\Vert_{X_{-\beta}\times X_{-\beta-\frac{1}{2}}}^2\leq \tau \Vert u_h^0\Vert_{X_{-\beta}\times X_{-\beta-\frac{1}{2}}}^2.$$
From (\ref{29}), we have that, for $k\in\mathbb{N}$,\\

$$\Vert u_h^0\Vert_{X_{-\beta}\times X_{-\beta-\frac{1}{2}}}^2=\Vert u_h^k\Vert_{X_{-\beta}\times X_{-\beta-\frac{1}{2}}}^2 +2\Delta t\sum_{j=0}^{k-1}(\Delta t)^2\Vert Au_h^{j+1}\Vert_{X_{-\beta}\times X_{-\beta-\frac{1}{2}}}^2$$

 $$ \hspace{4cm}+\Delta t \sum_{j=0}^{k-1}(\Delta t)^5\Vert A^2u_h^{j+1}\Vert_{X_{-\beta}\times X_{-\beta-\frac{1}{2}}}^2,$$
taking $k=k^*$, we deduce the existence of a positive constant $c_1$, which depends only on $T^*$ and $\delta$ such that \\

$$c_1 \Vert u_h^0\Vert_{X_{-\beta}\times X_{-\beta-\frac{1}{2}}}^2 \leq \Delta t\sum_{j=0}^{k^*-1}(\Delta t)^2\Vert Au_h^{j+1}\Vert_{X_{-\beta}\times X_{-\beta-\frac{1}{2}}}^2$$

$$\hspace{4.7cm} +\Delta t \sum_{j=0}^{k^*-1}(\Delta t)^5\Vert A^2u_h^{j+1}\Vert_{X_{-\beta}\times X_{-\beta-\frac{1}{2}}}^2.$$\\

Using the fact that $ H\subset X_{-\beta}\times X_{-\beta-\frac{1}{2}}$, with continuous embedding, we deduce the existence of a positive constant $c_2$, which depends only on $T^*$ and $\delta$ such that

\begin{equation}\label{31}
c_2 \Vert u_h^0\Vert_{X_{-\beta}\times X_{-\beta-\frac{1}{2}}}^2 \leq \Delta t\sum_{j=0}^{k^*-1}(\Delta t)^2\Vert Au_h^{j+1}\Vert_H^2 +\Delta t \sum_{j=0}^{k^*-1}(\Delta t)^5\Vert A^2u_h^{j+1}\Vert_H^2,
\end{equation}

holds uniformly with respect to $\Delta t>0$ for any solution of (\ref{12}) with initial data $u^0\in C_{\delta/ \Delta t}^\bot.$\\

Combining (\ref{28}) and (\ref{31}) yields Lemma 2.1, since $u_h$ and $u_l$ lie in orthogonal spaces with respect to the scalar product $\langle ., .\rangle_{X_{-\beta}}\times X_{-\beta-\frac{1}{2}}$ and $\langle A., A.\rangle_H$.
\end{proof}

\begin{rem}
The assumptions (\ref{14}) and (\ref{21})  or (\ref{17}) and (\ref{23}) hold true if we have (\ref{8}) ( see \cite{20}).
\end{rem}
\section{Polynomial stability via a "weakened" observability inequality and main result}
The main result of this paper reads as follows:\\
\begin{teo}
Assume that there exist positive constants $T$ and $c$ and $\beta>-\frac{1}{2}$ such that for all initial data $z^0\in D(A)=D(\mathcal{A})\times V$, we have (\ref{13}).\\
Then there exists $M>0$ such that
\begin{equation}\label{32}
E^k\leq\frac{M}{(1+t_k)^{\frac{1}{1+2\beta}}}\Vert z^0\Vert_{D(A)}^2, \;\;\;\forall k\geq 0,
\end{equation}
holds uniformly with respect to $0<\Delta t<1$,\; with $t_k=k\Delta t.$
\end{teo}

For the proof of this theorem, we need a technical lemma ( see Lemma 4.4 in \cite{2}).
\begin{lem}
Let $(\mathcal{E}^k)$ be a sequence of positive real numbers satisfying
$$\mathcal{E}_{k+1}\leq  \mathcal{E}_{k}-C \mathcal{E}_{k+1}^{2+\alpha}, \;\;\;\forall k\geq 0,$$
where $C>0$ and $\alpha>-1$ are constants. Then there exists a positive constant M (depending only on $C$ and $\alpha$) such that
$$\mathcal{E}_k\leq\frac{M}{(k +1)^{\frac{1}{\alpha+1}}},\;\forall k\geq 0.$$
\end{lem}
\begin{proof}

 We decompose the solution $z^k$ of (\ref{9}) as $z^k=w^k+u^k$ with $z^0=u^0$\\

 where $u^k$ is the solution of (\ref{12})  and $w^k$ is the solution of\\

\begin{equation}\label{33}
\left \{
\begin{array}{lcr}
            \frac{\tilde{w}^{k+1}-w^k}{\Delta t} = A\left(\frac{w^k+\tilde{w}^{k+1}}{2}\right)-BB^*\left(\frac{z^k+\tilde{z}^{k+1}}{2}\right),\;k\in \mathbb{N}, & \\\\

             \frac{w^{k+1}-\tilde{w}^{k+1}}{\Delta t}={(\Delta t)^2}{A^2}{w^{k+1}},\;k\in\mathbb{N},\\\\

             w^0=0.
\end{array}
\right.
\end{equation}\\\
\\\\\\\\\\\\\\\\\\

Applying Lemma 2.1 to $u^k = z^k - w^k$, we get: \\
$ c \Vert z^0\Vert_{X_{-\beta}\times X_{-\beta-\frac{1}{2}}}^2\leq 2\Bigg(\Delta t \displaystyle{\sum_{k\Delta t\in [0,T^*]}}\left\Vert B^*\left(\frac{z^k+\tilde{z}^{k+1}}{2}\right)\right\Vert_Y^2$\\
 $+ \Delta t \displaystyle{\sum_{k\Delta t\in [0,T^*]}}(\Delta t)^2\left\Vert Az^{k+1}\right\Vert_H^2 +\Delta t\displaystyle{\sum_{k\Delta t\in [0,T^*]}}(\Delta t)^5\left\Vert {A^2}{z}^{k+1}\right\Vert_H^2 \Bigg )$\\
 $ + 2\Bigg( \Delta t \displaystyle{\sum_{k\Delta t\in [0,T^*]}}\left\Vert B^* \left(\frac{w^k+\tilde{w}^{k+1}}{2}\right)\right\Vert_Y^2+ \Delta t \displaystyle{\sum_{k\Delta t\in [0,T^*]}}(\Delta t)^2\left\Vert Aw^{k+1}\right\Vert_H^2$\\
\begin{equation}\label{26}
+\Delta t\displaystyle{\sum_{k\Delta t\in [0,T^*]}}(\Delta t)^5\left\Vert {A^2}{w}^{k+1}\right\Vert_H^2 \Bigg).
\end{equation}
Now we follow the same approach as in the proof of Theorem 1.1 in \cite{8},  there exists a constant $G>0$ (independent of $\Delta t$) such that\\

$$\Delta t \displaystyle{\sum_{k\Delta t\in [0,T^*]}}\left\Vert B^* \left(\frac{w^k+\tilde{w}^{k+1}}{2}\right)\right\Vert_Y^2  + \Delta t \displaystyle{\sum_{k\Delta t\in [0,T^*]}}(\Delta t)^2\left\Vert Aw^{k+1}\right\Vert_H^2 $$\\

$\hspace{5cm}+ \Delta t\displaystyle{\sum_{k\Delta t\in [0,T^*]}}(\Delta t)^5\left\Vert {A^2}{w}^{k+1}\right\Vert_H^2$\\

$$\hspace{2cm}\leq G\Delta t  \displaystyle{\sum_{j\Delta t\in [0,T^*]}}\left\Vert B^* \left(\frac{z^j+\tilde{z}^{j+1}}{2}\right)\right\Vert_Y^2.$$

Combining this inequality and (\ref{26})), we get the existence of a constant $c$ such that\\
$$ c \Vert z^0\Vert_{X_{-\beta}\times X_{-\beta-\frac{1}{2}}}^2\leq \Delta t \displaystyle{\sum_{k\Delta t\in [0,T^*]}}\left\Vert B^* \left(\frac{z^k+\tilde{z}^{k+1}}{2}\right)\right\Vert_Y^2$$

 $$ \hspace{4cm}+ \Delta t \displaystyle{\sum_{k\Delta t\in [0,T^*]}}(\Delta t)^2\left\Vert Az^{k+1}\right\Vert_H^2$$

$$ \hspace{4cm}+\Delta t\displaystyle{\sum_{k\Delta t\in [0,T^*]}}(\Delta t)^5\left\Vert {A^2}{z}^{k+1}\right\Vert_H^2.$$

Combining this inequality and (\ref{11}), it follows that:

$$E^{l+1}\leq E^0- c \Vert z^0\Vert_{X_{-\beta}\times X_{-\beta-\frac{1}{2}}}^2.$$
By using a simple interpolation inequality ( see Proposition 2.3 in \cite{17}) and the fact that the function $E^k$ is nonincreasing, we obtain the existence of a constant $C>0$ such that
\begin{equation}\label{35}
E^{l+1}\leq E^0- C\frac{\{E^{l+1}\}^{2(1+ \beta)}}{\Vert z^0\Vert_{D(A)}^{2(1+2 \beta)}}.
\end{equation}
Estimate (\ref{35}) remains valid in successive intervals $[k(l+1), (k + 1)(l+1)]$, so, we have
$$E^{(k+1)(l+1)}\leq E^{k(l+1)}- C\frac{\{E^{(k+1)(l+1)}\}^{2(1+ \beta)}}{\Vert z^0\Vert_{D(A)}^{2(1+2 \beta)}}.$$
If we adopt the notation
$$\mathcal{H}^k=\frac{E^{k(l+1)}}{\Vert z^0\Vert_{D(A)}^2},$$
the last inequality gives
$$\mathcal{H}^{k+1}\leq\mathcal{H}^k- C(\mathcal{H}^{k+1})^{2(1+\beta)},\;\;\forall k\geq 0.$$
By using Lemma 3.1, we obtain the existence of a constant $M>0$ such that
$$\mathcal{H}^k\leq \frac{M}{(1+k)^{\frac{1}{1+2\beta}}},\;\forall k\geq 0,$$
and consequently, for all $\Delta t <1$, we have
$$E^{k(l+1)}\leq \frac{M}{(1+t_k)^{\frac{1}{1+2\beta}}}\Vert z^0\Vert_{D(A)}^2,\;\forall k\geq 0,$$
which obviously implies (\ref{32}).
\end{proof}
\section{Applications}

\subsection{Two coupled wave equations}
 We consider the following system
\begin{equation*}
\left\{
\begin{array}{lcr}
           u_{tt}(x,t)-u_{xx}(x,t)+\alpha y(x,t)=0,\;\;0<x<1,\;t>0,\\\\

           y_{tt}(x,t)-y_{xx}(x,t)+\alpha u(x,t)+\gamma y_t(x,t)=0,\;\;0<x<1,\;t>0,\\\\

           u(0,t)=u(1,t)=y(0,t)=y(1,t)=0,\;\;t>0,\\\\

           u(x,0)=u_0(x),\;u_t(x,0)=u_1(x),\; y(x,0)=y_0(x),\;y_t(x,0)=y_1(x),\;\;0<x<1.\end{array}
\right.
\end{equation*}
with $\gamma >0$ and $\alpha>0$ small enough. Take $H=L^2(0,1)^4$, the operator $B$ defined by
$$B=\left(\begin{array}{l}
0\\
0\\
0\\
 \sqrt{\gamma}\end{array}\right),\;\;\;(B^*=(0,0,0,\sqrt{\gamma}))$$
which is a bounded operator from $Y=L^2(0,1)$ into $H$ and the operator $A$ as follows
$$A=\left( \begin{array}{c}
0\;\;\;\;\;\;\;0\;\;\;\;\;\;\;\;I\;\;\;\;\;\;\;0 \\
0\;\;\;\;\;\;\;0\;\;\;\;\;\;\;\;0\;\;\;\;\;\;\;I\\
\partial_{xx}\;\;\;\;-\alpha\;\;\;\;0\;\;\;\;\;\;0\\
-\alpha\;\;\;\;\;\;\partial_{xx}\;\;\;\;\;0\;\;\;\;\;\;0
\end{array} \right),$$\\\\

with $D(A)=(H_0^1(0,1)\cap H^2(0,1))^2\times L^2(0,1)^2$.\\

If $\alpha$ is small enough, namely if $\alpha <\pi^2$, this operator $A$ is a skew-adjoint\\

operator in $H$, then the above system is equivalent to system (\ref{6}) where $Z=\left(\begin{array}{l}
u\\
y\\
u_t\\
y_t\end{array}\right)$.
We use the same method in \cite{1}, we show that the eigenvalues of $A$  are
$$sp(A)=\{i\mu_{+,k}\}\cup\{i\mu_{-,k}\}\;\;k\in\mathbb{Z^*},$$
with $\mu_{+,k}=\sqrt{\alpha+k^2\pi^2}$,\; $\mu_{-,k}=\sqrt{-\alpha+k^2\pi^2}$ and $\mu_{+,-k}=-\mu_{+,k}$\;\; $\mu_{-,k}=-\mu_{-,k},\;\forall k\in\mathbb{N^*}.$\\

The corresponding eigenvectors are, respectively, given by
$$w_{+,k}=\left(\begin{array}{l}
\frac{1}{i\mu_{+,k}}\sin(k\pi x)\\
\frac{1}{i\mu_{+,k}}\sin(k\pi x)\\
\sin(k\pi x)\\
\sin(k\pi x)\end{array}\right),\;\;w_{-,k}=\left(\begin{array}{l}
\frac{1}{i\mu_{-,k}}\sin(k\pi x)\\
-\frac{1}{i\mu_{-,k}}\sin(k\pi x)\\
\sin(k\pi x)\\
-\sin(k\pi x)\end{array}\right)\;\;k\in\mathbb{Z^*},$$
with $w_{+,-k}=w_{+,k}$  and  $w_{-,-k}=w_{-,k}$,\;\;$\forall\;k\in\mathbb{N^*}$.\\

(\ref{17}) is satisfied and (\ref{23}) holds with $\beta=0$ (see \cite{1} for more details), thus the above system is weakly observable \cite{20}, and consequently polynomially stable. Now, according to Theorem 3.1 we have
\begin{pro}
The solutions of
\begin{equation*}
\left \{
\begin{array}{lcr}
            \frac{\tilde{Z}^{k+1}-Z^k}{\Delta t} = A\left(\frac{Z^k+\tilde{Z}^{k+1}}{2}\right)-BB^*\left(\frac{Z^k+\tilde{Z}^{k+1}}{2}\right),\;k\in \mathbb{N}, & \\\\

             \frac{Z^{k+1}-\tilde{Z}^{k+1}}{\Delta t}={(\Delta t)^2}{A^2}{Z^{k+1}} ,\;k\in\mathbb{N}, \\\\

             Z^0=(u_0,y_0,u_1,y_1).
\end{array}
\right.
\end{equation*}
are polynomially uniformly decaying in the sense of (\ref{32}) with $\beta=0$.
\end{pro}
\subsection{Two boundary coupled wave equations}

We consider the following system
\begin{equation*}
\left\{
\begin{array}{lcr}
           u_{tt}(x,t)-u_{xx}(x,t)=0,\;\;0<x<1,\;t>0,\\\\

           y_{tt}(x,t)-y_{xx}(x,t)+\gamma y_t(x,t)=0,\;\;\;0<x<1,t>0,\\\\

           u(0,t)=y(0,t)=0,\;\;\;t>0,\\\\

           y_x(1,t)=\alpha u(1,t),\;\; t>0,\\\\

           u(x,0)=u_0(x),\;u_t(x,0)=u_1(x),\; y(x,0)=y_0(x),\;y_t(x,0)=y_1(x),\;\;\;\;0<x<1.\end{array}
\right.
\end{equation*}
when $\alpha,\beta\in\mathbb{R}$ with $\gamma>0$ and $\alpha>0$ small enough. Hence it is written\\

in the form (\ref{6}) with the following choices: Take $H=L^2(0,1)^4$, the \\

operator $B$ as follows:
$$B=\left(\begin{array}{l}
0\\
0\\
0\\
 \sqrt{\gamma}\end{array}\right),\;\;\;(B^*=(0,0,0,\sqrt{\gamma}))$$

which is a bounded operator from $Y=L^2(0,1)$ into $H$ and the operator\\

$A$ defined by

$$D(A)=\{(u,y,u_t,y_t)\in (V\cap H^2(0,1))^2\times L^2(0,1)^2: y_x(1)=\alpha u(1); u_x(1)=\alpha y(1)\}$$

when $V=\{v\in H^1(0,1); v(0)=0\}$ and

$$AZ=\left( \begin{array}{l}
u_t \\
y_t\\
u_{xx}\\
y_{xx}
\end{array} \right),$$

when $Z=\left(\begin{array}{l}
u\\
y\\
u_t\\
y_t\end{array}\right)$.

If $\alpha$ is small enough, namely if $\alpha <1$, this operator $A$ is skew-adjoint in $H$.\\

As in \cite{1}, the eigenvalues of $A$ are
$$sp(A)=\{i\mu_{+,k}\}\cup\{i\mu_{-,k}\}\;\;k\in\mathbb{Z^*},$$
with $\mu_{+,k}=\frac{\pi}{2}+k\pi +\epsilon_{+,k}$,\; $\mu_{-,k}=\frac{\pi}{2}+k\pi -\epsilon_{-,k}$ and $\mu_{+,-k}=-\mu_{+,k}$\;\; $\mu_{-,k}=-\mu_{-,k},\;\forall k\in\mathbb{N^*},$\\

where $\epsilon_{+,k}=\arctan(\frac{\alpha}{\mu_{+,k}})$ and $\epsilon_{-,k}=\arctan(\frac{\alpha}{\mu_{-,k}}).$\\

The corresponding eigenvectors are, respectively, given by
$$w_{+,k}=\left(\begin{array}{l}
-\frac{1}{i\mu_{+,k}}b_{+,k}\sin(\mu_{+,k}.)\\
\frac{1}{i\mu_{+,k}}b_{+,k}\sin(\mu_{+,k}.)\\
-b_{+,k}\sin(\mu_{+,k}.)\\
b_{+,k}\sin(\mu_{+,k}.)
\end{array}\right),\;\;w_{-,k}=\left(\begin{array}{l}
\frac{1}{i\mu_{-,k}}b_{-,k}\sin(\mu_{-,k}.)\\
\frac{1}{i\mu_{-,k}}b_{-,k}\sin(\mu_{-,k}.)\\
b_{-,k}\sin(\mu_{-,k}.)\\
b_{-,k}\sin(\mu_{-,k}.)\end{array}\right)
\;\;k\in\mathbb{Z^*},$$

with $w_{+,-k}=w_{+,k}$  and  $w_{-,-k}=w_{-,k}$,\;\;$\forall\;k\in\mathbb{N^*}$,\\

 and where $b_{+,k}$ and $b_{-,k}$ are chosen to normalize the eigenvectors.\\\\\\\\

(\ref{17}) is satisfied and (\ref{23}) holds with $\beta=0$ (see \cite{1} for more details), thus the above system is weakly observable \cite{20}, and consequently polynomially stable. Now, applying  Theorem 1.3  we get

\begin{pro}
The solutions of
\begin{equation*}
\left \{
\begin{array}{lcr}
            \frac{\tilde{Z}^{k+1}-Z^k}{\Delta t} = A\left(\frac{Z^k+\tilde{Z}^{k+1}}{2}\right)-BB^*\left(\frac{Z^k+\tilde{Z}^{k+1}}{2}\right),\;k\in \mathbb{N}, & \\\\

             \frac{Z^{k+1}-\tilde{Z}^{k+1}}{\Delta t}={(\Delta t)^2}{A^2}{Z^{k+1}} ,\;k\in\mathbb{N}, \\\\

             Z^0=(u_0,y_0,u_1,y_1).
\end{array}
\right.
\end{equation*}
are polynomially uniformly decaying in the sense of (\ref{32}) with $\beta=0$.
\end{pro}
\section{Further comments}
\begin{enumerate}
\item
As we mentioned in the introduction, our methods and results require the assumption that the damping operator
$B$ is bounded. We use the fact that the polynomial
decay of the energy is a consequence of the observability properties of the conservative system. That is the case, even in the continuous setting. However, in several relevant applications when the feedback law is unbounded \cite{2}, our method does not apply.
\item
Another drawback of our method is that it is restrictive for a class of operators, that is the spectrum of the  operator $A$ associated
with the undamped problem satisfies such a gap condition. This is due to the method we employ, which is based on a  discrete Ingham type inequalities. One could ask if we have some results about polynomial stability for the time semi-discrete scheme when the following generalized gap condition $w_{k+N}-w_k\geq N\gamma,\;k\in\mathbb{N}$, is satisfied for $N\geq 3$. To our knowledge, we don't have a discrete Ingham type inequalities when the last gap condition is verified, and this issue is widely open.\\

In our context, it would be also relevant to ask if our methods allow to deal with stabilization properties of fully discrete approximation scheme with   numerical viscosity or under a suitable CFL type condition on the time and space discretization parameters as in the exponential case \cite{8}.
\item
Other question arise when discretizing in time semilinear wave equations. For instance, in \cite{12}, under suitable properties of the nonlinearity it is proved that the  polynomial decay property of solutions holds. It would be interesting to analyze
whether the same polynomial decay property holds, uniformly with respect to the time-step, for the numerical schemes
analyzed in this article in this semilinear setting.
\end{enumerate}

\bibliographystyle{plain}

\end{document}